\documentclass [11pt,oneside,a4paper,mathscr]{amsart}
\usepackage{geometry}
\newgeometry{margin=1.1in}

\usepackage{amsfonts, amsmath, amssymb, amsthm,mathtools, stmaryrd}
\usepackage{verbatim}
\usepackage[normalem]{ulem}
\usepackage{tikz-cd}
\usepackage{enumitem}

\usepackage{longtable} %tables with pagebreak
\usepackage[center]{caption}
\usepackage{diagbox}

\usepackage{array}

\usepackage{hyperref}
\usepackage{cleveref}

\theoremstyle{plain}
\newtheorem{thm}{Theorem}[section]
\newtheorem{cor}[thm]{Corollary}

\newtheorem{prop}[thm]{Proposition}

\numberwithin{equation}{section}

\newtheorem{conjecture}{Conjecture}  % Conjecture labeled by letters
\newtheorem{conj}[conjecture]{Conjecture}  % Conjecture labeled by letters

\theoremstyle{definition}

\newtheorem{example}[thm]{Example}
\newtheorem{lemma}[thm]{Lemma}
\newtheorem{rmk}[thm]{Remark}
\newtheorem{question}[thm]{Question}

\theoremstyle{remark}

\newcommand{\BA}{{\mathbb{A}}}

\newcommand{\BC}{{\mathbb{C}}}

\newcommand{\BL}{{\mathbb{L}}}

\newcommand{\BQ}{{\mathbb{Q}}}

\newcommand{\BZ}{{\mathbb{Z}}}

\newcommand{\CC}{{\mathcal C}}

\newcommand{\CH}{{\mathcal H}}

\newcommand{\CO}{{\mathcal O}}

\DeclareFontFamily{OT1}{rsfs}{}
\DeclareFontShape{OT1}{rsfs}{n}{it}{<-> rsfs10}{}
\DeclareMathAlphabet{\curly}{OT1}{rsfs}{n}{it}

\newcommand{\p}{\mathbb{P}}

\newcommand{\Mbar}{{\overline M}}
\newcommand{\vir}{{\text{vir}}}
\newcommand{\Chow}{\mathrm{Chow}}

\newcommand{\PT}{\mathsf{PT}}

\newcommand{\DT}{\mathsf{DT}}

\newcommand{\VW}{\mathsf{VW}}

\newcommand{\PTcal}{\mathcal{PT}}
\newcommand{\PTcalprim}{\mathcal{PT}^{\mathrm{prim}}}
\newcommand{\DTcal}{\mathcal{DT}}

\newcommand{\Hilb}{\mathsf{Hilb}}

\newcommand{\Sym}{\mathrm{Sym}}

\newcommand{\ch}{\mathsf{ch}}

\newcommand{\Exp}{{\mathrm{Exp}}}
\newcommand{\Log}{{\mathrm{Log}}}

\newcommand{\GV}{{\mathsf{GV}}}

\begin{document}
\baselineskip=14.5pt
\title[Towards refined curve counting on the Enriques surface II]{Towards refined curve counting on the Enriques surface II: Motivic refinements}

\author{Georg Oberdieck}

\address{Universit\"at Heidelberg, Institut f\"ur Mathematik}
\email{georgo@uni-heidelberg.de}
\date{\today}

\begin{abstract}
We study the motivic Pandharipande-Thomas invariants of the Enriques Calabi-Yau threefolds in fiber curve classes by basic computations and analysis of a wallcrossing formula of Toda. Motivated by our results we conjecture a formula for the perverse Hodge numbers of the compactified Jacobian fibration of linar systems on Enriques surfaces in terms of its Betti numbers. This leads to an asymptotic for said Hodge numbers and
raises questions about the behaviour of the extremal Hodge numbers.
\end{abstract}

\maketitle

\setcounter{tocdepth}{1} 
\tableofcontents

\section{Introduction}
Let $Y$ be an Enriques surface and let $X \to Y$ be the associated covering K3 surface with covering involution $\tau : X \to X$. There are two Calabi-Yau threefolds into which $Y$ naturally embedds.
The first is the local Enriques surface defined as the total space of the canonical bundle, $K_Y= \mathrm{Tot}(\omega_Y)$. It can also be constructed as the quotient
\[ K_Y = (X \times \BC)/\langle \tau \times (-1)_{\BC} \rangle. \]
The second is the Enriques Calabi-Yau threefold
\[
Q = (X \times E)/\langle \tau \times (-1)_{E} \rangle,
\]
where $E$ is an elliptic curve.
By projecting to the second factor these threefolds admit isotrivial K3 fibrations
\[ \pi : K_Y \to \BC/\langle (-1)_{\BC} \rangle \cong \BC,
\quad \pi : Q \to E/\langle (-1)_{E} \rangle \cong \p^1 \]
with a double Enriques fibers over $0 \in \BC$ and $4$ double Enriques fibers over the torsion points of $E$ respectively. This leads to the embeddings
\[ \iota_{K_Y} : Y \cong (X \times 0)/\BZ_2 \hookrightarrow K_Y, \quad
\iota_{Q,a} : Y \cong (X \times a)/\BZ_2 \hookrightarrow Q \]
where $a \in E$ is $2$-torsion. The first embedding is of course just the zero section.
We often drop the subscript $a$ in the second embedding.

Pandharipande-Thomas (PT) theory associates curve counting invariants to Calabi-Yau threefolds $X$ by integrating over the moduli space of stable pairs (where\footnote{In this paper we will work with integral cohomology always {\em modulo torsion}. We write $H^k(X,\BZ)$ and $H_k(X,\BZ)$ for the integral (co)homology modulo torsion. All curve classes will be considered modulo torsion.} $\beta \in H_2(X,\BZ)$),
\[
\PT_{n,\beta}(X) = \int_{[ P_{n,\beta}(X) ]^{\vir}} 1 \quad \in \BZ.
\]
A curve class $\beta \in H_2(Q,\BZ)$ is a fiber class if $\pi_{\ast} \beta = 0$ or equivalently, if $\beta$ is the pushforward of a class via $\iota_{Q}$.
By a degeneration argument \cite{MP}, the PT invariants of $Q$ for fiber curve classes are essentially the same as $K_Y$. Concretely, one has
\[
\sum_{n} 
\sum_{\beta \in H_2(Y,\BZ)} \PT_{n,\iota_{Q \ast} \beta}(Q) (-p)^n Q^{\beta}
=
\left( \sum_{n} \sum_{\beta \in H_2(Y,\BZ)} \PT_{n,\iota_{K_Y\ast} \beta}(K_Y) (-p)^n Q^{\beta} \right)^4.
\]
Moreover, the invariants $\PT_{n,\iota_{K_Y \ast}\beta}(Q)$ were all computed in \cite{Enriques}.
%and were computed recently in \cite{Enriques}.
They are equivalent to the Gromov-Witten invariants of the Enriques surface $Y$ with a $(-1)^{g-1} \lambda_{g-1}$ Hodge insertion.

If a Calabi-Yau threefold $X$ admits the action of a torus $T$, the Pandharipande-Thomas invariants are {\em refined} by
the $T$-equivariant $K$-theoretic PT invariants introduced by Nekrasov-Okounkov \cite{NO}.
These invariants are no longer integers, but rational functions in the equivariant weights of the torus.
For $K_Y$ these "NO-refined" PT invariants were studied in \cite{RefinedEnriques1} and an explicit conjecture in terms of Jacobi forms was given for them.

Here we are interested in a second, different refinement of PT invariants, namely the motivic PT invariants introduced by Kontsevich-Soibelman \cite{KS} and Bussi-Joyce-Meinhardt \cite{BJM}. The invariant (the virtual motive of the moduli space of stable pairs) takes values in some Grothendieck group of varieties and can be considered for all Calabi-Yau threefolds (regardless of the existence of a torus action). The Euler characteristic specialization of the virtual motive then recovers the original PT invariant whenever the Calabi-Yau threefold is compact\footnote{It recovers the Behrend-function weighted Euler characteristic of the moduli space, which is equal to the degree of the virtual class if the moduli space is compact by work of Behrend \cite{Behrend}.}. In the non-compact case, however, this may fail. For example, the motivic invariants of $K_Y$ do not recover the usual PT invariants under specialization and hence they are not the right quantity to consider. The motivic refinement is also much more delicate to work with: it requires an orientation data compatible under summands (constructed canonically by Joyce-Upmeier \cite{JU}) and is not invariant under deformation of $X$. We will usually consider the Hodge and Betti polynomials of the virtual motive.
Explicit computations of motivic PT invariants in other geometries can be found in \cite{BBS, MMNS, DM15, KKP, CKK, Dec22a, Leigh}.

The goal of this paper is then to study the Hodge polynomial of the motivic Pandharipande-Thomas invariants of $Q$ in fiber curve classes for a generic Enriques surface $Y$. The invariants are denoted $\PTcal_{n,\beta}$ for $\beta \in H_2(Y,\BZ) \subset H_2(Q,\BZ)$,
and give a refinements of the usual PT invariants $\PT_{n,\beta}(Q)$.
We obtain the following series of results:
\begin{enumerate}
	\item[(a)] We compute the invariants $\PTcal_{0,df}$, where $f$ is a half-fiber of an elliptic fibration $Y \to \p^1$ on the Enriques in Proposition~\ref{prop:PTcal for fiber classes} as
\[ \sum_{d \geq 0} \PTcal_{df,0} q^d
=
\prod_{m \geq 1} \frac{(1-q^{2m})^{6}}{(1-(t \tilde{t})^{-1} q^{2m}) (1-q^{m})^{8} (1- t \tilde{t} q^{2m})}. \]
%	This corresponds to the fiber class of the induced elliptic fibration $Q \to (E \times \p^1)/\BZ_2$.
	\item[(b)] We moreover give an explicit conjecture for $\PTcal_{n,df}$ for all $n,d$, see Conjecture~\ref{conj:fiber classes complete}.
	This predicts all PT invariants for the fiber classes of the induced elliptic fibration $Q \to (E \times \p^1)/\BZ_2$.
	\item[(c)] We state Toda's equation, which links PT invariants with generalized DT invariants of sheaves supported in fibers of $\pi :Q \to \p^1$. We also discuss the dependence of the generalized DT invariants on the Chern character, see Section~\ref{subsec:structure results}. These structure results for PT invariants were very useful in the work \cite{Enriques}.
	\item[(d)] For $2 \nmid \beta$, we write the generating series of PT invariants as a sum of two terms: the first is essentially the contribution from $K_Y$ and of pure Hodge type, the second is a new contribution which vanishes under the unrefined limit and can be thought to come from the compactified Jacobian of linear systems on the Enriques surface.
	The first term controls a certain asymptotic of the motivic PT invariants.
\end{enumerate}
Computation (a) in particular shows that the motivic PT invariants of $Q$ (after the $\chi_{-y}$-genus specialization $\tilde{t}=1$) differ from the Nekrasov-Okounkov refined PT invariants of $K_Y$.

Despite the above results, a complete conjectural picture for the motivic theory of $Q$ in fiber classes
is unfortunately missing so far
and requires additional computational tools.

\subsection{Gopakumar-Vafa invariants}
The motivic PT invariants of a Calabi-Yau threefold are conjecturally related to a refined version of Maulik-Toda theory \cite{MT} (which is a mathematical proposal for Gopakumar-Vafa invariants following earlier proposals by \cite{HST,KiemLi}). For the CY3 $Q$ and fiber curve classes $\beta$ with $2 \nmid \beta$, this boils down to a connection between motivic PT theory and the geometry of the relative compactified Jacobian of linear systems on the Enriques surface.
The result (d) above has consequences for the topology of these moduli spaces which we now discuss.

Let $\beta \in H_2(Y,\BZ)$ be an effective curve class on a Enriques surface $Y$ and let $M_{\beta}$ be one of the two connected components of the moduli space of $1$-dimensional stable sheaves $F$ on $Y$ with Euler characteristic $\chi(F)=1$.
We assume that $Y$ is generic and that $2\nmid \beta$,
in which $M_{\beta}$ is a smooth projective variety of dimension $\beta^2+1$.
Moreover, by a result of Sacca \cite{Sacca}, $M_{\beta}$ is Calabi-Yau, i.e. has trivial canonical bundle $\omega_{M_{\beta}} \cong \CO$.
By sending a sheaf to its fitting support, we obtain
the Hilbert-Chow morphism $\rho : M_{\beta} \to \p$ to a linear system of curves in class $\beta$.
We let 
\[ {^ph}^{i,j}(M_{\beta}) := 
h^j( {^p\CH^{i}}(R \rho_{\ast} \BQ[\dim M_{\beta}]) )  \]
be the assoiated perverse Hodge numbers,
which are the graded dimension of the induced perverse filtration on $H^{\bullet}(M_{\beta},\BQ))$, see \cite{dCM} for an introduction to perverse sheaves.
If $\beta^2=2d$ and $\beta$ is effective, then $M_{\beta}$ has dimension $2d+1$ and the Hilbert-Chow morphism $\rho:M_{\beta} \to \p^d$ has $(d+1)$-dimensional fibers \cite{Sacca}.
Hence ${^ph}^{i,j}(M_{\beta})$ is non-zero only for
$i \in \{ -(d+1), \ldots, d+1 \}$ and $j \in \{ - d, \ldots, d\}$ .

The perverse Hodge numbers recover the usual Betti numbers by
\begin{equation} \label{Betti in perverse Hodge}
b_{k+\dim(M_{\beta})}(M_{\beta}) = \sum_{i+j=k} {^ph}^{i,j}(M_{\beta}).
\end{equation}
If $\beta_1^2=\beta_2^2$, then $M_{\beta_1}, M_{\beta_2}$ are known to be birational by \cite{Beckmann,NY}, so their Betti (and in fact the Hodge) numbers agree by Kontsevich's theory of $p$-adic integration. In other words, the Betti numbers of $M_{\beta}$ depend upon $\beta$ only through the square $\beta^2=2d$. We write $b_{i,d} = b_i(M_{\beta})$.
% for these Betti numbers.

Our main conjecture is as follows:

\begin{conj} \label{conj:perverse Hodge intro}
(a) If $2 \nmid \beta$, then the perverse Hodge numbers ${^ph}^{i,j}(M_{\beta})$ also depend upon $\beta$ only through the square $\beta^2$. If $\beta^2 = 2d$, we write
\[ {^ph}^{i,j}(M_{\beta}) = {^ph}^{i,j}_d. \]
(b) We have the following identity
	\begin{small}
		\begin{multline}
			\label{Key equation}
			\sum_{d \geq 0} \sum_{i,j} {^ph}^{i,j}_d (-1)^{i+j} p^i u^j q^d
			= 
			\frac{(1-u^{-1}p)(1-up)}{(-p)} \prod_{m \geq 1} \frac{1}{(1-q^m)^{8}} \\
			\times 
			\prod_{\substack{m \geq 1 \\ m \text{ odd}}} 
			\frac{1}{(1-u^{-2} q^m) (1-u^2 q^m) (1-up q^m) (1-u p^{-1} q^m)(1-u^{-1} p q^m) (1-u^{-1} p^{-1} q^m) (1-q^m)^2} \\
			-
			\left( \sum_{d=0}^{\infty} u^{-(2d+1)} \sum_{i=0}^{2d+2} (-u)^i b_{i,d} \right)
						\prod_{m \geq 1} \frac{(1-u^2 q^{2m})(1-u^{-2} q^{2m}) (1-q^{2m})^2}{(1-up q^{2m}) (1-u^{-1} p^{-1} q^{2m}) (1-u^{-1} p q^{2m})(1-u p^{-1} q^{2m})}
		\end{multline}
	\end{small}
\end{conj}

The first term on the right in \eqref{Key equation} is explicit
and is related to the K-theoretic PT theory of $K_Y$ \cite{RefinedEnriques1}.
The second term in \eqref{Key equation} is inexplicit since it depends on the Betti numbers of $M_{\beta}$ (which are not known). So the conjecture says that the Betti numbers of the
$M_{\beta}$'s determines their perverse Hodge numbers.
% of the $M_{\beta}$'s.
This is sursprising given that perverse Hodge numbers of $M_{\beta}$ refine the Betti numbers via \eqref{Betti in perverse Hodge}.

The first few Betti numbers of $M_{\beta}$ were determined by Sacca.
They are 
$1,0,10,22,10,0,1$ for $M_{\beta}$ if $\beta^2=2$, and $b_1(M_\beta)=0$ and $b_2(M_\beta)=11$ for all $\beta^2 >  2$.
%Otherwise the Betti numbers are unknown.
Inserting these values we can read off already a good chunk of the first perverse Hodge numbers of $M_{\beta}$.
The results are listed in tables 1-4 below. 
The question mark indicate numbers which are not determined.
The numbers in Table 2 have been verified by A. Leraand.

%\begin{figure}[h]
\begin{minipage}{0.25\linewidth}
	\begin{longtable}{| l | c |}
		\hline
		\diagbox[width=1.1cm,height=0.7cm]{$i$}{$j$} 
		& $0$  \\
		\nopagebreak \hline
		$-1$ & $1$  \\
		$0$ & $2$ \\
		$1$ & $1$ \\
		\hline
		\caption{${^ph}^{i,j}_0$}
	\end{longtable}
\end{minipage}
\begin{minipage}{0.35\linewidth}
	\begin{longtable}{| c | c  c  c |}
		\hline
		\diagbox[width=1.1cm,height=0.7cm]{$i$}{$j$} 
		& $-1$ & $0$ & $1$ \\
		\nopagebreak \hline
		$-2$ & $1$ &  & $1$  \\
		$-1$ &  & $8$ &  \\
		$0$ & $1$ & $22$ & $1$ \\
		$1$ &  & $8$ &  \\
		$2$ & $1$ &  & $1$ \\
		\hline
		\caption{${^ph}^{i,j}_1$}
	\end{longtable}
\end{minipage}
\begin{minipage}{0.38\linewidth}
	\begin{longtable}{| c | c  c  c c c |}
		\hline
		\diagbox[width=1.1cm,height=0.7cm]{$i$}{$j$} 
		& $-2$ & $-1$ & $0$ & $1$ & $2$ \\
		\nopagebreak \hline
		$-3$ & $1$ &  & $1$ & & $1$ \\
		$-2$ &  & $9$ & & $9$ &  \\
		$-1$ & $1$ & $2$ & $47$ & $2$ & $1$ \\
		$0$ & ? & ? & ? & ? & ? \\
		$1$ & $1$ & $2$ & $47$ & $2$ & $1$ \\
		$2$ &  & $9$ & & $9$ &  \\
		$3$ & $1$ &  & $1$ & & $1$ \\
		\hline
		\caption{${^ph}^{i,j}_2$}
	\end{longtable}
\end{minipage}

\begin{minipage}{0.99\linewidth}
	\begin{longtable}{| c | c  c  c c c c c |}
		\hline
		\diagbox[width=1.1cm,height=0.7cm]{$i$}{$j$} 
		& $-3$ & $-2$ & $-1$ & $0$ & $1$ & $2$ & $3$ \\
		\nopagebreak \hline
		$-4$ & $1$ &  & $1$ & & $1$ & & $1$ \\
		$-3$ &  & $9$ & & $10$ & & $9$ & \\
		$-2$ & $1$ &  & $55$ & & $55$ & & $1$ \\
		$-1$ &  & $10$ & $22$ & $220$ & $22$ & $10$ & \\
		$0$ & ? & ? & ? & ? & ? & ? & ? \\
		$1$ &  & $10$ & $22$ & $220$ & $22$ & $10$ & \\
		$2$ & $1$ &  & $55$ & & $55$ & & $1$ \\
		$3$ &  & $9$ & & $10$ & & $9$ & \\
		$4$ & $1$ &  & $1$ & & $1$ & & $1$ \\
		\hline
		\caption{${^ph}^{i,j}_3$}
	\end{longtable}
\end{minipage}

The second interesting feature of \eqref{Key equation} is that the second (so far unknown) term only contributes to part of the formula. Asymptotically, for $d \to \infty$ and $i+d, j+d$ fixed it does not contribute, and so we can read of the asymptotics of the perverse Hodge numbers:
%More precisely we have the following. Denote

Consider the shifted perverse Hodge numbers
\[ {^p\widetilde{h}}^{i,j}_d = {^ph}^{-(d+1)+i,-d+j}_d. \]

\begin{prop} \label{prop:asymptotics}
	If \Cref{conj:perverse Hodge intro} holds
	and  $i<d/2$ and $j<d/2-1$, then ${^p\widetilde{h}}^{i,j}_d$
	%:=	{^ph}^{-(d+1)+i,-d+j}_d$ 
	is independent of $d$. 
	Let $\tilde{h}^{ij}_{\infty}$ denote this constant.
	Then
	\begin{align*}
	\sum_{i,j \geq 0} \tilde{h}^{ij}_{\infty} x^i y^j
	& =
	(1-xy) \prod_{n \geq 1} \frac{1}{(1-x^{n+1} y^{n-1}) (1-x^{n-1} y^{n+1}) (1-x^n y^n)^{10}} \\
	& = 1 + x^{2} + 9 x y + y^{2} + x^{4} + 10 x^{3} y + 56 x^{2} y^{2} + 10 x y^{3} + y^{4} \\
	& \quad + x^{6} + 10 x^{5} y + 66 x^{4} y^{2} + 276 x^{3} y^{3} + 66 x^{2} y^{4} + 10 x y^{5} + y^{6} + O((x,y)^8).
	%\\
	%& x^{8} + 10 x^{7} y + 67 x^{6} y^{2} + 341 x^{5} y^{3} + 1177 x^{4} y^{4} + 341 x^{3} y^{5} + 67 x^{2} y^{6} + 10 x y^{7} + y^{8} + O(x, y)^{10}
\end{align*}	
\end{prop}

This proposition is parallel to the asymptotics of peverse Hodge numbers of compactified Jacobians of linear systems on Fano surfaces \cite{PSSZ}.
% $\p^2$ \cite{ShenPi} or Fano surfaces \cite{SZ}.

We state also the implied asymptotic for the Betti numbers:
\begin{cor} Assume \Cref{conj:perverse Hodge} holds, $2 \nmid \beta$ and $\beta^2=2d$. Then $b_i(M_{\beta}) = b_i^{\infty}$ whenever $d > i+1$,
	where
	% of $M_{\beta}$ stabilize for $d \to \infty$. Let $b^i_{\infty} = \lim_{d \geq \infty} b^i(M_d)$. Then
	\begin{align*}
		\sum_{i \geq 0} b_i^{\infty} x^i
		& = (1-x^2) \prod_{n \geq 1} \frac{1}{(1-x^{2n})^{12}} \\
		& = 1 + 11x^{2} + 78x^{4} + 430x^{6} + 2015x^{8} + 8373x^{10} + 31706x^{12} +\ldots
	\end{align*}
\end{cor}
\begin{proof}
	This follows from Proposition~\ref{prop:asymptotics} by setting $x=y$ and using $b_{k,d} = \sum_{i+j=k} {^p\widetilde{h}}^{i,j}_d$.
\end{proof}

The behaviour of the extremal perverse Hodge numbers $ {^p\widetilde{h}}^{i,0}_d$ (which appear in the left most column in Tables 1-4) points to behaviour which can be observed also in Lagrangian fibrations of hyperk\"ahler manifolds \cite{SY}.
We state this as an basic question to investigate.

\begin{question}
Do we have the following evaluation?
\[ {^p\widetilde{h}}^{i,0}_d
=
\begin{cases} 1 & \text{ if } i = 0, 2,4, \ldots, 2d+2 \\
	0 & \text{ else }
\end{cases}. \]
\end{question}

\subsection{Plan of the paper}
In Setion~\ref{sec:intro motivic refined invariants} we give some background on the motivic PT invariants of a Calabi-Yau threefolds
and their relationship to Maulik-Toda theory.
In Section~\ref{sec:motivic invariants} we consider the motivic invariants of the Enriques Calabi-Yau threefolds. We first state Toda's wallcrossing formula (Theorem~\ref{thm:motivic toda formula}) and then discuss the implications for motivic PC invariants in the cases $\beta = df$ and $2 \nmid \beta$ respetively.
Section~\ref{sec:open questions} discusses two more open questions.

\subsection{Acknowledgements}
Part of this work was written during a visit of the author to the Newton Institute in Cambridge. I thank the institute for a pleasant working environment. I also thank Arkadij Bojko, Nikolas Kuhn and Junliang Shen for useful discussions on refined curve counting. The author was supported by the starting grant 'Correspondences in enumerative geometry: Hilbert schemes, K3 surfaces and modular forms', No 101041491 of the European Research Council.

\section{Motivic invariants of Calabi-Yau threefolds}
\label{sec:intro motivic refined invariants}
Let $X$ be a projective Calabi-Yau threefold and let $\beta \in H_2(X,\BZ)$ be curve class.

There are two main sheaf-theoretic constructions of curve counting invariants: Pandharipande-Thomas invariants using the moduli space of stable pairs \cite{PT}, and Maulik-Toda invariants using the moduli space $\Mbar_{\beta}$ of stable $1$-dimensional sheaves \cite{MT}.
These two sets of invariants were conjectured in \cite{MT} to be equivalent.
We first recall this conjecture, and then explain how it can be refined on both sides (following Migliorini-Shende, Maulik-Yun \cite{MigShende1,MigShende2,MaulikYun}).

\subsection{Unrefined Gopakumar-Vafa polynomials}
Let $\PT_{n,\beta} = \int_{[ P_{n,\beta}(X) ]^{\vir}} 1 \in \BQ$ be the (unrefined) stable pair invariant of $X$ and consider the generating series
\[
Z_{PT}(X) = \sum_{n,\beta} \PT_{n,\beta} (-p)^n Q^{\beta}.
\]
We define the unrefined Gopakumar-Vafa "polynomials" $\GV^{\mathrm{unref}}_{\beta}(u,p)$ by the equality
\[
\Log\, Z_{PT}(X)
=
\frac{(-p)}{(1-p)^2} \sum_{\beta} \GV^{\mathrm{unref}}_{\beta}(u,p) Q^{\beta}
\]

Conjecturally the $\GV^{\mathrm{unref}}_{\beta}(u,p)$
are Laurent polynomials invariant under $p \mapsto 1/p$, see \cite[Sec.3.4]{PT}.
For Calabi-Yau threefolds satisfying the GW/PT correspondence, this is known due to the reecnt proof of the finiteness of BPS invariants \cite{DIW}.

\begin{rmk}
The Gopakumar-Vafa invariants $n_{g,\beta}$ (as defined by PT theory) can be extracted from the  $\GV^{\mathrm{unref}}_{\beta}(u,p)$ by the expansion
\[ \GV^{\mathrm{unref}}_{\beta}(u,p) = \sum_{g} n_{g}^{\beta} (-p)^{-g} (1-p)^{2g}, \]
or equivalently,
\[
(\Log Z_{PT}(X))|_{p=-q} = \sum_{\beta} Q^{\beta}
\sum_{g} n_g^{\beta} (q^{-1/2} + q^{1/2})^{2g-2},
\]
 see \cite[Sec.3.4]{PT}.
We will not need these invariants here.
\end{rmk} 

Let $M_{\beta,1}$ be the moduli space of stable $1$-dimensional sheaves $F$ with class $\ch_2(F) = \beta$ and Euler characteristic $1$. Let $\rho : M_{\beta,1} \to \mathrm{Chow}_{\beta}$ be the map to the Chow variety.
Let $\Phi$ be the perverse sheaf on $M_{\beta,1}$ associated to a Calabi-Yau orientation data \cite{MT}. 

Maulik and Toda conjecture the following relationship:

\begin{conj}(Maulik-Toda \cite{MT}) \label{conj:MT} We have the equality:
	\[
	\GV_{\beta}^{\mathrm{unref}}(p) = 
	\sum_{i} \chi( {^p\CH^i}(R \rho_{\ast} \Phi) ) (-p)^i 
	%= \sum_{g \geq 0} n_{g}^{\beta} (q^{1/2} + q^{-1/2})^{2g}.
	\]
\end{conj}

\begin{rmk}
Equivalently one has
$\sum_{i} \chi( {^p\CH^i}(R \pi_{\ast} \Phi) ) q^i = \sum_{g \geq 0} n_{g}^{\beta} (q^{1/2} + q^{-1/2})^{2g}$.
\end{rmk}

We give two computations of Gopakumar-Vafa invariants:
\begin{example}
Let $C$ be a isolated smooth curve of genus $g$ in $X$ of class $\beta$.
Then $\chi(\CO_C) = 1-g$, so its contribution to the PT invariants is
\[ \sum_n e(\Sym^n(C)) (-1)^n q^{1-g+n} = q^{1-g} \frac{1}{(1+q)^{e(C)}} = q^{1-g} (1+q)^{2g-2}. \]
On the other hand, $M_{\beta,1}$ is smooth of dimension $g$ for sheaves supported on $C$, so we simply have $\Phi = \BQ[g]$ and
$\CH_{p}^{i}(R \pi_{\ast} \Phi) = H^{i+g}(J(C))$ and 
\[ \sum_{i} \chi( {^p\CH^{i}}(R \pi_{\ast} \Phi) ) q^i = \sum_{i=-g}^{g} \dim H^{i+g}(J(C)) q^i =
q^{-g} (1+q)^{2g}. \]
\end{example}

%\begin{example}
%For local $\p^2$, one has
%$\GV^{\mathrm{unref}}_{\beta=1} = 3$,
%$\GV^{\mathrm{unref}}_{\beta=2} = -6$,
%$\GV^{\mathrm{unref}}_{\beta=3} = 10 p^{-1} + 7 + 10 p^{-1}$,
%so  $n_{0}^{\beta=3}=27$, $n_{1}^{\beta=3} = -10$,
%see e.g. \cite{CKK}.
%% $n_{0}^{\beta=1} = 3$,
%%$n_{0}^{\beta=2} = -6$, $n_{0}^{\beta=3}=-10$, $n_{1}^{\beta=3} = 27$.
%\end{example}

\begin{example}
For the local Enriques surface $K_Y$ we
can compute the Gopakumar-Vafa polynomials by taking $\Log$ of 
\cite[Theorem 1.2]{RefinedEnriques1} and then use 
\cite[Eqn (2.1)]{RefinedEnriques1}. This gives
%\eqref{unrefined PT} and using \eqref{Exp2}. 
%This gives us:
\[
\GV_{\beta}^{\mathrm{unref}}(p)
=
\begin{cases}
a(\beta^2/2) & \text{ if } 2 \nmid \beta \\
a(\beta^2/2) - \frac{1}{2} a((\beta/2)^2/2) & \text{ if } 2 \mid \beta,
\end{cases}
\]
where
\[ \sum_{n} a(n) q^n = 
2 \prod_{\substack{m \geq 1 \\ m \text{ odd}}} \frac{1}{(1-p^{-1} q^m)^2 (1-q^m)^{4} (1-p q^m)^2}
\prod_{m \geq 1} \frac{1}{(1-q^m)^{8}}
\]
For example,   $n_1^{f} = 2$, $n_{1}^{s+f}=32$, $n_{2}^{s+f} = -4$.
In particular, $\GV_{\beta}^{\mathrm{unref}}$ and hence $n_{g}^{\beta}$ 
only depend upon $\beta$ through the square $\beta^2$ and whether $\beta$ is $2$-divisible or not.
\end{example}

\subsection{Motivic classes}
The motivic invariants we will consider will take value in the Grothendieck
ring 
$\mathsf{K}_{\text{var}}^{\widehat{\mu}}[L^{-1}]$
of varieties equipped with an action of 
groups of $n^{th}$ roots
of unity and localized at the class of the affine line $L=[\BA^1]$, see \cite{BJM}.
%\[ K^{\widehat{\mu}}_{\mathrm{var}}[L^{-1}] \]
%of varieties carrying actions of groups of $n^{th}$ roots
%of unity $\mathsf{K}_{\text{var}}^{\widehat{\mu}}$
%localized at the inverse of class of the affine line
%$L=[\BA^1]$.
In this ring $\mathsf{K}_{\text{var}}^{\widehat{\mu}}[L^{-1}]$
the class $L$ has a canonical square root given by
\[ L^{\frac{1}{2}}=1-[\mu_2,\rho] \]
whee $\rho$ is the canonical action of
$\mu_2=\{\pm1\}$ on itself, see \cite[Sec.2]{KKP}.
%The ring $\mathsf{K}^{\widehat{\mu}}_{\mathrm{var}}[L^{-1}]$ 
%therefore contains all powers of $L^{\pm 1/2}$.

The ring $\mathsf{K}_{\text{var}}^{\widehat{\mu}}[L^{-1}]$ admits natural specialization to numerical invariants.
For example, there is a Hodge specialization $\chi_{t,\tilde{t}} : \mathsf{K}_{\text{var}}^{\widehat{\mu}}[L^{-1}] \to \BQ[t^{\pm 1}, \tilde{t}^{\pm 1}]$, see \cite{BBS},
which on a smooth projective variety $V$ takes values
\[ \chi_{t,\tilde{t}}(V) = \sum_{p,q} (-1)^{p+q} h^{p,q}(V) t^p \tilde{t}^q. \]
It also satisfies
\[ \chi_{t,\tilde{t}}(L^{1/2}) = - (t \tilde{t})^{1/2}. \]

A further specialization is the Betti polynomial $B = \chi_{t,\tilde{t}}|_{\tilde{t}=1}$, so that
\[ H(V) = \chi_{t,\tilde{t}}(V)|_{t=\tilde{t}=u} = \sum_{i} b_i(V) (-u)^i, \]
or the $\chi_y$ specialization
\[ \chi_{-t}(V) = \chi_{t,\tilde{t}}(V)|_{\tilde{t}=1} = \sum_{i} (-t)^i \chi(V, \wedge^i \Omega_V). \]
In particular, we get
$H(L^{1/2}) = -u$ and $\chi_{-t}(L^{1/2}) = -t^{1/2}$.

\begin{rmk} If a motive lies in the subring generated by $L^{1/2}$, then the $\chi_{-t}$-realization and the Betti realization coincide under the variable change $t^{1/2}=u$. This happens for the motivic invariants of local $\p^2$ \cite{CKK}, and is responsible for the fact that NO-refined PT invariant for local $\p^2$ agree (conjecturally so far) with the Betti realization of the motivic PT invariants \cite[Conj.3.1]{KPS}.
\end{rmk}

\subsection{Motivic PT invariants}
We let
\[
\PT^{\mathrm{mot}}_{n,\beta} = [P_{n,\beta}(X)]^{\mathrm{vir}} \in K^{\widehat{\mu}}_{\text{var}}[L^{-1}]
\]
denote the motivic PT invariant of $X$
as defined by Bussi-Joyce-Meinhardt \cite{BJM}.
We will also use motivic generalized DT invariants as defined in \cite{BBBJ}.
These invariants depend on a choice of orientation data, which was recently constructed in a canonical way by \cite{JU}.
The orientation data of \cite{JU} is moreover 'compatible under direct sums', which implies that we have a motivic wall-crossing formulas in the sense of Kontsevich-Soibelman \cite{KS}. 

Since we assumed $X$ to be compact, the Euler characteristic specialization of the motivic PT invariant is the (unrefined) PT invariant \cite{BJM}:
\[
\chi( \PT^{\mathrm{mot}}_{n,\beta}  ) = \PT_{n,\beta}.
\]
% we obtain the Behrend function:
%\[
%\chi([M]^{\vir}) = H([M]^{\vir})|_{u=1} = (-1)^{\dim(M)} \chi(M).
%\]

We will use the following special case: If $M$ is a smooth moduli space of stable sheaves on $X$ and the critical structure is given by the smooth ambient space $M$ with critical function $f=0$, then we have the motivic invariant
%
%Let $M$ be a moduli space which is a critical locus of a function.
%There is an associated motive $[M]^{\vir}$.
%If $M$ is smooth, and we take the function $f=0$ on $M$, then 
\begin{equation} \label{Mvir for M smooth}
	 [ M ]^{\vir} = L^{-\dim(M)/2} [M]. 
 \end{equation}
Since $H(L^{1/2}) = -u$, this yields the Betti realization is
\[ H([M]^{\vir}) = \sum_{i=0}^{2 \dim(M)} b_i(M) (-u)^{i-\dim(M)}. \]

We also have the wall-crossing factor
	$\chi_{t, \tilde{t}}([ \p^{n-1} ]^{\vir}) = (-1)^{n-1} [ n ]_{t \tilde{t}}$.

\subsection{Refined Gopakumar-Vafa polynomials}
%Let $\PTmot_{n,\beta}$ be the motivic stable pair invariant.
Consider the Betti realization of the generating series of motivic PT invariants:
\[ Z_{PT}^{\mathrm{Betti}}(X) = \sum_{n,\beta} H(\PT^{\mathrm{mot}}_{n,\beta}) (-p)^n Q^{\beta}. \]
The refined Gopakumar-Vafa "polynomials" $\GV_{\beta}(u,p)$ are defined by
\[
\Log Z_{PT}^{\mathrm{Betti}}(X)
=
\frac{(-p)}{(1-u^{-1}p)(1-up)} \sum_{\beta} \GV_{\beta}(u,p) Q^{\beta}
\]

Work of Migliorini-Shende \cite{MigShende1, MigShende2} and Maulik-Yun \cite{MaulikYun} (see Example~\ref{ex:MS MY}) showed how the connection to Maulik-Toda theory can be refined.
Consider again the moduli space $M_{\beta,1}$ and the Hilbert-Chow map $\rho : M_{\beta,1} \to \Chow_{\beta}$. 
We define the perverse Hodge numbers
\[ {^ph^{ij}_{\beta}} := 
h^j( {^p\CH^{i}}(R \pi_{\ast} \Phi) ).
\]

\begin{conj}[{\cite{MigShende1, MigShende2, MaulikYun}}]\footnote{The author learned this conjecture from discussions with J. Shen.}
	\label{conj:refined MT}
\begin{align*} 
	\mathrm{GV}_{\beta}(u,p)
	& =
	\sum_{i,j} {^ph}^{i,j}_{\beta} (-1)^{i+j} p^i u^j
\end{align*}
\end{conj}

\begin{rmk} This conjecture refines the Maulik-Toda Conjecture~\ref{conj:MT}. Indeed, setting $u=1$ we get
\[
\mathrm{GV}_{\beta}(u=1,p) 
%= \sum_{i} (-p)^i \sum_{j} (-1)^j h^j( {^p\CH^{i}}(R \pi_{\ast} \Phi) )
=
\sum_{i} (-p)^i \chi( {^p\CH^{i}}(R \pi_{\ast} \Phi) )
=
\mathrm{GV}^{\mathrm{unref}}_{\beta}(p)
\]
\end{rmk}

\begin{example} \label{ex:MS MY}
Let $C$ be a smooth isolated genus $g$ curve on $X$.
The contribution of $C$ to the PT generating series in class $[C]$ is
\begin{align*}
\sum_{n \geq 0} 
H([\Sym^n(C)]^{\vir}) (-p)^{1-g+n}
%& = (-p)^{1-g} \sum_{n \geq 0} H(L^{-n/2} [C^{(n)}]) (-p)^n \\
& = (-p)^{1-g} \sum_{n \geq 0} (u^{-1} p)^n \sum_{i} b_i(C^{(n)}) (-u)^i \\
& = \frac{(-p)^{1-g} (1-p)^{2g}}{(1-u^{-1}p)(1-up)}.
\end{align*}
%q^{1-g} \frac{(1-q)^{2g}}{(1-u^{-1} q) (1-u q)}
%\]
Hence we get
\[ \GV_{[C]} = (-p)^{-g} (1-p)^{2g} = q^{-g} (1+q)^{2g} = (q^{-1/2} + q^{1/2})^{2g}, \quad q=-p, \]
which matches the Betti numbers of the Jacobian of a genus $g$ curve.
\end{example}
\begin{example}
Let $\pi : \CC \to B$ be a family of integral plane curves of genus $g$ over a smooth base $B$. Assume that the relative compactified Jacobian $\rho : \overline{JC}=M \to B$ is smooth. Let $\pi^{[d]} : \CC^{[d]} \to B$ be the $d$-th relative Hilbert scheme of points, which is then also smooth.
Under these conditions, Migliorini-Shende and Maulik-Yun \cite{MigShende1,MaulikYun} showed that
(see \cite[page 4]{MigShende2} for the precise form) 
\[
\sum_{d=0}^{\infty} q^d
R \pi^{[d]}_{\ast} \BQ[g+\dim(B)] = 
\frac{\sum_{i} q^i \ {^p\CH^{i-g}}(R \rho_{\ast} \BQ[\dim M]) [g-i]}{(1-q ) (1-qL) } .
\]
%where $[ - ]$ is the cohomological shift and $( - )$ is the Tate-twist.

\begin{comment} 
\begin{rmk}
By \cite{MigShende2} one has
\[
\sum_{d=0}^{\infty} q^d
R \pi^{[d]}_{\ast} \BQ = 
\frac{\sum_{i} q^i \ {^p\tilde{\CH}^{i}}(R \rho_{\ast} \BQ) [-i]}{(1-q) (1-q L) } 
\]
where ${^p\tilde{\CH}^{i}}$ is the convention of \cite{MigShende2}, which satisfies ${^p\tilde{\CH}^{i}}(F) = {^p\CH^{i+\dim(B)}}(F)[ - \dim(B) ]$ with respect to the usual convention.
(For once, if $F$ is a local system, then
${^p\tilde{\CH}^{i}}(F) = F$ if $i=0$ and $0$ otherwise.
On the other hand,
${^p\CH^{i}}(F[\dim B]) = F[\dim B]$ if $i=0$ and zero otherwise.)
Also note that ${^p\CH^{i}}(F) = {^p\CH^{i-j}}(F[j])$. Combining these both statements reduces to the above equation to the previously claimed equation.
\end{rmk}
\end{comment}

Taking cohomology and Betti numbers one finds
\[
\sum_{d=0} q^d \sum_i (-u)^i b_{i+g+\dim(B)}(\CC^{[d]})
=
\sum_{d=0} q^d (-u)^{-g-\dim(B)} H( \CC^{[d]} )
=
\frac{\sum_{i} q^i \sum_{j} (-u)^j {^ph}^{i-g,j+g-i}_{\beta} }{(1-q) (1-q u^2) },
\]
and hence
\begin{align*}
\sum_{d=0} (-p)^{d+1-g} H( [\CC^{[d]}]^{\vir} )
& =
(-p)^{1-g} (-u)^g \sum_{d} (u^{-1} p)^d (-u)^{-g-\dim(B)} H(\CC^{[d]}) \\
& =
(-p)^{1-g} (-u)^{g} \frac{\sum_{i,j} (u^{-1}p)^i (-u)^j {^ph}^{i-g,j+g-i}_{\beta} }{(1-u^{-1} p)(1-up)} \\
& =
\frac{(-p)}{(1-u^{-1}p)(1-up)}
\sum_{i,j} p^i u^j (-1)^{i+j} {^ph}^{i,j}.
\end{align*}

For statements involving also reducible curves, see \cite{MigShende2}.
\end{example}

%\begin{example}
%Let $X=\mathrm{Tot}(K_{\p^2})$. In degree $1$, the moduli space $P_{n,\beta}$ is a $\p^{n-1}$-bundle over the linear system $\p^2$, so we get
%\[ [ P_{n,\beta} ]^{\vir} = L^{-(n+1)/2} [\p^2] [\p^{n-1}] = [\p^2]^{\vir} [(\p^1)^{(n-1)}]^{\vir}. \]
%Thus we get
%\[
%\sum_{n \geq 0} (-p)^n H(  [ P_{n,\beta} ]^{\vir} ) = H( [\p^2]^{\vir}) \cdot
%\frac{(-p)}{(1-u^{-1}p)(1-up)},
%\]
%and hence $\GV_{\beta=1}(u,p) = u^{-2} + 1 + u^2$.
%See \cite{CKK} for more computations.
%%which has Poincare polynomial
%%\[ H(  [ P_{n,\beta} ]^{\vir} ) = (-u)^{-(n+1)} (1+u^2 + u^4) (1 + u^2 + \ldots + u^{2(n-1)} \]
%\end{example}

\begin{rmk}
	Following Choi-Katz-Klemm \cite{CKK},
	one can also define refined Gopakumar-Vafa invariants $N_{j_L, j_R}^{\beta}$ by the equation
	\begin{align*}
		& (\Log Z_{PT}^{\mathrm{Betti}}(X)) \\
		= & \frac{-p}{(1 -up) (1 - u^{-1}p)}
		\sum_{\beta} Q^{\beta} \sum_{j_L, j_R} N_{j_L,j_R}^{\beta} (-1)^{2 (j_L+j_R)}
		\sum_{m_R=-j_R}^{j_R} u^{-2m_R} \sum_{m_L} p^{-2 m_L} 
		%= &
		%\frac{q}{(1 + q u) (1 + q u^{-1})}
		%\sum_{\beta} Q^{\beta} \sum_{j_L, j_R} N_{j_L,j_R}^{\beta} (-1)^{2 (j_L+j_R)}
		%\sum_{m_R=-j_R}^{j_R} u^{-2m_R} \sum_{m_L} (-q)^{-2 m_L} \\
	\end{align*}
	We will not need this notion here.
	\qed
\end{rmk}

\section{Motivic invariants of the Enriques Calabi-Yau threefold}
\label{sec:motivic invariants}
Let $Y$ be a generic Enriques surface and consider the Enriques Calabi-Yau threefold
\[ Q = (X \times E)/\langle \tau \times (-1) \rangle,  \]
where $\tau : X \to X$ is the involution on the covering K3 surface.

For every $2$-torsion point $a \in E$,
% we have the inclusion $\iota_{Q,a} : Y \cong (X \times a)/\BZ_2 \hookrightarrow Q$. 
the pushforward in cohomology 
\[ \iota_{Q, a} : H^{\ast}(Y,\BZ) \to H^{\ast}(Q,\BZ) \]
is independent of $a$ modulo torsion. Hence we drop $a$ from notation, and often view a cohomology class $v \in H^{\ast}(Y,\BZ)$ as a cohomology class on $Q$ via pushforward along $\iota_{Q}$,  $v= \iota_{Q\ast} v \in H^{\ast}(Q,\BZ)$.
% Hence we will henceforth view cohomology classes (modulo torsion) on $Y$ as classes on $Q$.
For $\beta \in H_2(Y,\BZ)$, we let $\PT_{n,\beta}^{\mathrm{mot}}$ be the motivic PT invariant of $Q$, and for $v \in H^{\ast}(Y,\BZ)$, we let
$\DT^{Q,\mathrm{mot}}(v)$ be the generalized motivic Donaldson-Thomas invariant of $Q$ counting semi-stable sheaves $F$ on $Q$ with Chern character $\ch(F) = v$. 
%(by which we mean $\ch(F) = \iota_{Q,a \ast} v$, 
In particuclar, $F$ is supported on fibers of $\pi : Q \to E/\BZ_2$.

We denote the $\chi_{t, \tilde{t}}$ specializations of these motivic invariants by
\begin{gather*}
	\PTcal_{n,\beta} = \chi_{t,\tilde{t}}(\PT_{n,\beta}^{Q, \mathrm{mot}}) \\
	\DTcal(v) = \chi_{t,\tilde{t}}(\DT^{Q,\mathrm{mot}}(v)).
\end{gather*}

We will also consider "BPS invariants" $\Omega(v)$ defined by
\[
\DTcal(v) = \sum_{\substack{k|v \\ k \geq 1}} \frac{1}{k [k]_{t \tilde{t}}} \Omega(v/k)|_{t \mapsto t^k, \tilde{t} \mapsto \tilde{t}^k}.
\]
The classes $\Omega(v)$ are expected to have good integrality properties \cite{JS}.

\subsection{Structure results} \label{subsec:structure results}
Parallel to \cite[Sec.4.2]{RefinedEnriques1} we have the following two structure results.

\begin{thm}[Motivic Toda's equation] \label{thm:motivic toda formula} \begin{align*}
\sum_{\beta} \sum_{n \in \BZ} \PTcal_{n,\beta} (-p)^n Q^{\beta}
=& \prod_{\substack{r \geq 0 \\ \beta > 0 \\ n \geq 0 }} \exp\left( (-1)^{r-1} [n+r]_{t \tilde{t}} \DTcal(r,\beta,n) Q^{\beta} p^n \right) \\
 \times&  \prod_{\substack{r>0 \\ \beta>0 \\ n>0}} \exp\left( (-1)^{r-1} [n+r]_{t \tilde{t}} \DTcal(r,\beta,n) Q^{\beta} p^{-n} \right).
 \end{align*}
\end{thm} 

In the following statement we refer to \cite[Thm.4.4]{RefinedEnriques1} 
for the definition of the type.
\begin{thm}[Dependence of DT invariants] \label{thm:dependence of DT invariants}
The invariant $\DTcal(v)$ only depends upon $v=(r,\beta,n)$ through the square $v^2=\beta^2-2rn-r^2$, the divisibility $\mathrm{div}(v)$ and the type of $v$.
\end{thm} 

We can rewrite Toda's equation in terms of the BPS classes $\Omega(v)$:
\begin{equation} \label{Log toda eqn}
\begin{aligned}
\Log\left( \sum_{\beta} \sum_{n \in \BZ} \PTcal_{n,\beta} (-p)^n Q^{\beta} \right) = \\
\sum_{r,n \geq 0, \beta>0} (-1)^{r-1} [n+r]_{t \tilde{t}} \Omega(r,\beta,n) Q^{\beta} p^{n} 
& - \sum_{\substack{v=(r,\beta,n)\\  2|v, r/2 \text{ odd} \\
		r,n \geq 0, \beta>0}} \left[ \frac{n+r}{2} \right]_{(t \tilde{t})^2} \Omega(v/2)|_{t \mapsto t^2, \tilde{t} \mapsto \tilde{t}^2} Q^{\beta} p^n \\
+ \sum_{r,n > 0, \beta>0} (-1)^{r-1} [n+r]_{t \tilde{t}} \Omega(r,\beta,n) Q^{\beta} p^{-n} 
& - \sum_{\substack{v=(r,\beta,n)\\  2|v, r/2 \text{ odd} \\
		r,n >0, \beta>0}} \left[ \frac{n+r}{2} \right]_{(t \tilde{t})^2} \Omega(v/2)|_{t \mapsto t^2, \tilde{t} \mapsto \tilde{t}^2} Q^{\beta} p^{-n}.
\end{aligned}
\end{equation}

We will consider these equations below in two examples:
\begin{enumerate}
	\item Fiber classes $\beta = df$ for $f$ a half-fiber.
	\item $2 \nmid \beta$.
\end{enumerate}

\subsection{Fiber curve classes} \label{subsec:fiber curve classes}
Let $Y \to \p^1$ be an elliptic fibration on $Y$ with half-fiber $f$.
For curve class $\beta=df$
we have an isomorphism of moduli spaces, compare
\cite[Sec.4.5]{RefinedEnriques1}:
\[
P_{df,0}(Q) \cong \Hilb^n([\p^1 \times E]/\BZ_2) = \Hilb^n(\p^1 \times E)^{\BZ_2}.
\]
where $\BZ_2$ acts on $\p^1 \times E$ as a product of an involution on $\p^1$ and multiplication by $-1$ on $E$. In particular, $P_{df,0}(Q)$ is smooth. We also assume below that its critical struture is trivial (so we are in the situation \eqref{Mvir for M smooth}).
Then we have the following computation:
%This allows us to compute the motivic invariants in this case:
\begin{prop} \label{prop:PTcal for fiber classes}
\[ \sum_{d \geq 0} \PTcal_{df,0} q^d
=
\prod_{m \geq 1} \frac{(1-q^{2m})^{6}}{(1-(t \tilde{t})^{-1} q^{2m}) (1-q^{m})^{8} (1- t \tilde{t} q^{2m})}. \]
\end{prop}
\begin{proof}
We have more generally the following. Let $S$ be a smooth surface with a $\BZ_2$-action with $\ell$ isolated singularieties. Let $\tilde{S}$ the resolution of the quotient $S/\BZ_2$.
% $M$ is smoooth, let $[M]^{\mathrm{vir}} = L^{-\dim(M)/2} [M]$ be the virtual motive. 
Then we claim that
\[
\sum_{n=0}^{\infty} [\Hilb^n(S)^{\BZ_2}]^{\mathrm{vir}} q^n
=
\left( \prod_{i \geq 1} \frac{(1-q^{2i})^2}{1-q^i} \right)^{d}
\cdot \Exp\left( \sum_{i \geq 1} q^{2i} [\tilde{S}]^{\mathrm{vir}} \right),
\]
where we use the power structure on the right of motives (see \cite{BBS}) and used the trivial critical structure (so $[M]^{\mathrm{vir}}$ is just $L^{-\dim(M)/2} [M]$).

The claim implies the proposition, since applying $\chi_{t, \tilde{t}}$ is compatible with the power structure (see \cite{BBS}), so gives
\[
\sum_{n=0}^{\infty} \chi_{t,\tilde{t}}([\Hilb^n(S)^{\BZ_2}]^{\mathrm{vir}}) q^n
=
\left( \prod_{i \geq 1} \frac{(1-q^{2i})^2}{1-q^i} \right)^{d}
\cdot \Exp\left( \sum_{i \geq 1} q^{2i} \chi_{t,\tilde{t}}([\tilde{S}]^{\mathrm{vir}}) \right).
\]
Moreover, for over $\BZ_2$ action on $\p^1 \times E$, we have the resolution of the quotient
$(\p^1 \times E)/\BZ_2$ is a rational elliptic surface with virtual motive $L^{-1} + 10 + L$.

To prove the claim, one uses the following result of \cite[Sec.4]{GLM}:
\[
\sum_{n=0}^{\infty} [\Hilb^n(S)^{\BZ_2}] q^n
=
\left( \prod_{i \geq 1} \frac{(1-q^{2i})^2}{1-q^i} \right)^{d}
\cdot \Exp\left( \sum_{i \geq 1} q^{2i} \BL^{i-1} [\tilde{S}] \right).
\]
Correcting for the $L$-shift in dimension then boils down to making the variable change $q^{2} \mapsto L^{-1} q^{2i}$ in the second term only, see \cite{GLM} for details on the geometry. (Near the fixed points one has the decomposition \cite[(4.2)]{RefinedEnriques1}.)
	%this is correct for the contributions away from the fixed points and at the fixed points one uses the decomposition \eqref{erfer3}, 
\end{proof}

\begin{cor} \label{cor:DTcal in fiber classes}
	With the above assumptions, for $r \geq 1$ we have
\[ \DTcal(r,df,0) =
\begin{cases}
\frac{8}{r [r]_{t \tilde{t}}} & \text{ if } r \text{ is odd}, r|d \\
\frac{-2}{r [r]_{t \tilde{t}}} \Big( (t \tilde{t})^{-r/2} - 2 + (t \tilde{t})^{r/2} \Big) & \text{ if } r \text{ is even}, r|d \\
0 & \text{ else}
\end{cases}
\]
and thus
\[
\Omega(r,df,0) =
\begin{cases}
8 & \text{ if } r=1 \\
\chi_{t,\tilde{t}}([\p^1]^{\vir} )
= 
- ((t \tilde{t})^{1/2} + (t \tilde{t})^{-1/2} )
& \text{ if } r=2, 2|d \\
0 & \text{ else}
\end{cases}
\]
Moreover, $\DTcal(r,df,n) = \Omega(r,df,n) = 0$ if $r,n>0$.
\end{cor}
\begin{proof}
The first part follows from the motivic Toda equation, Proposition~\ref{prop:PTcal for fiber classes} and a computation.
For $r,n>0$, the vanishing follows as in \cite[Cor.4.5]{RefinedEnriques1}.
% one gets:
%\[ \DT(r,df,n) = 0 \text{ whenever } r,n>0. \]
\end{proof}

By the dependence property of DT invariants (Theorem~\ref{thm:dependence of DT invariants}) we also have
\begin{gather*}
\DTcal(0,2df,1) = \DTcal(0,0,1) = \chi_{t,\tilde{t}}([Q]^{\mathrm{vir}}) \\
\DTcal(0,(2d+1)f,1) = \DTcal(0,f,0) = 8 \chi_{t,\tilde{t}}([E]^{\mathrm{vir}}),
\end{gather*}
where $E$ is a smooth elliptic curve.
In \cite{Toda2} a $\chi$-independence conjecture for generalized Donaldson-Thomas invariants counting $1$-dimensional sheaves on Calabi-Yau threefolds was made. In our special case it would say that:
\[
\chi(\DTcal(0,\beta,n)) = \sum_{k \geq 1, k|\mathrm{gcd}(\beta,n)} \frac{1}{k^2} \chi(\DTcal(0,\beta/k,1)).
\]
Maulik-Shen \cite{MS} proved a refined $\chi$-independence for local $\p^2$, which suggests that the $\chi$-independence should hold in a refined way. Thus it is naturally to expect that:
\[
\DTcal(0,\beta,n) = \sum_{k \geq 1, k|\mathrm{gcd}(\beta,n)} \frac{1}{k [k]_{t \tilde{t}}} 
\DTcal(0,\beta/k,1)|_{t \mapsto t^k, \tilde{t} \mapsto \tilde{t}^k}.
\]
Assuming the $\chi$-independence hence fully determines the $\DTcal(0,df,n)$. 

Using Toda's equation we thus arrived at the following conjectural evaluation of the invariants $\PTcal_{n,df}$.
\begin{conj} \label{conj:fiber classes complete}
\begin{multline*}
\sum_{d \geq 0} \sum_{n \in \BZ} \PTcal_{n,df} (-p)^n q^d
=
\prod_{m \geq 1} \frac{(1-q^{2m})^{6}}{(1-(t \tilde{t})^{-1} q^{2m}) (1-q^{m})^{8} (1- t \tilde{t} q^{2m})} \\
\cdot \Exp\left( 
\frac{-p}{(1- (t \tilde{t})^{1/2} p)(1- (t \tilde{t})^{-1/2} p)}
\left[
\sum_{\substack{d \geq 1 \\ d \text{ odd}}} 8\chi_{t,\tilde{t}}([E]^{\vir}) q^d 
+
 \sum_{\substack{d \geq 1 \\ d \text{ even}}} \chi_{t,\tilde{t}}([Q]^{\vir}) q^d
\right] \right)
\end{multline*}
\end{conj}

Finally, we consider also the implication for the Gopakumar-Vafa polynomials.
Consider the plethystic logarithm:
\begin{multline*}
\Log \sum_{d \geq 0} \sum_{n \in \BZ} H(\PTcal_{n,df}) (-p)^n q^d
=
\frac{-p}{(1- up)(1-u^{-1}p)} 
\left[ 
\sum_{\substack{d \geq 1 \\ d \text{ odd}}} 8H([E]^{\vir}) q^d
+
\sum_{\substack{d \geq 1 \\ d \text{ even}}} H([Q]^{\vir}) q^d \right] \\
+ \sum_{\substack{m \geq 1 \\ m \text{ even}}} (u^{-2} + 2 + u^2) q^m
+ \sum_{\substack{m \geq 1 \\ m \text{ odd}}} 8 q^m
\end{multline*}
This yields conjecturally the Gopakumar-Vafa polynomials
\begin{align*}
\GV_{df}(p,q)
& = 
\begin{cases}
	8 H([E]^{\vir}) - 8 \frac{(1- up)(1-u^{-1}p)}{p} & \text{ if } d \text{ odd} \\
	H([Q]^{\vir}) - (u^{-2} + 2 + u^2) \frac{(1- up)(1-u^{-1}p)}{p} & \text{ if } d \text{ even}
\end{cases} \\
& = 
\begin{cases}
	8 ( -p^{-1} + 2 - p ) & \text{ if } d \text{ odd} \\
	-u^{2} p - u^{2} p^{-1} - 8 u - 2 p + 24 - 2 p^{-1} - 8 u^{-1} - u^{-2} p - u^{-2} p^{-1} & \text{ if } d \text{ even}
\end{cases}
%\begin{cases}
%H([Q]^{\vir}) - 8 \frac{(1- up)(1-u^{-1}p)}{p} & \text{ if } d \text{ odd} \\
%8 H([E]^{\vir}) - (u^{-2} + 2 + u^2) \frac{(1- up)(1-u^{-1}p)}{p} & \text{ if} d \text{ even}.
%\end{cases}
\end{align*}

Using Conjecture~\ref{conj:refined MT},
i.e. the expansion
$ \mathrm{GV}_{\beta}(u,p) =
\sum_{i,j} {^ph}^{i,j}_{\beta} (-1)^{i+j} p^i u^j$,
we obtain the perverse Hodge numbers of $M_{df,1}$:\\

\begin{center}
\begin{minipage}{0.35\linewidth}
	\begin{longtable}{| l | c |}
		\hline
		\diagbox[width=1.1cm,height=0.7cm]{$i$}{$j$} 
		& $0$  \\
		\nopagebreak \hline
		$-1$ & $8$  \\
		$0$ & $16$ \\
		$1$ & $8$ \\
		\hline
		\caption{${^ph}^{i,j}(M_{df,1}(Q))$ \text{ for } d \text{ odd}}
	\end{longtable}
\end{minipage}
\begin{minipage}{0.35\linewidth}
	\begin{longtable}{| c | c  c  c c c |}
		\hline
		\diagbox[width=1.1cm,height=0.7cm]{$i$}{$j$} 
		& $-2$ & $-1$ & $0$ & $1$ & $2$ \\
		\nopagebreak \hline
		$-1$ & 1 & & $2$ & & $1$ \\
		$0$ & & $8$ & $24$ & $8$ & \\
		$1$ & $1$ &  & $2$ & & $1$  \\
		\hline
		\caption{${^ph}^{i,j}(M_{df,1}(Q))$ \text{ for } d \text{ even}}
	\end{longtable}
\end{minipage}
\end{center}
Note that for $d$ even, the perverse Hodge numbers ${^ph}^{i,j}(M_{df,1}(Q))$ equal the perverse Hodge numbers of the fibration $Q \to [(\p^1 \times E)/\BZ_2]$.

Refined Gopakumar-Vafa invariants have been computed for elliptically fibered compact Calabi-Yau threefolds reduced fibers via physics methods in \cite{HKKrefined}.
However, the above is the first conjectural computation for an elliptic fibration with non-reduced fibers.
%It is remarkable that $\GV_{df}(p,q)$ only depend on the parity of $d$.

%\prod_{m \geq 1} \frac{(1-q^{2m})^{16}}{(1-q^m)^{8}} \Exp\left( \sum_{m \geq 1} 

\subsection{Curve classes $2 \nmid \beta$}
We may formally subtract contributions from multiple and reducible curves from PT invariants by defining the invariants $\PTcalprim_{n,\beta}$ by
\[ \sum_{n,\beta} \PTcalprim_{n,\beta} (-p)^n Q^{\beta} = 
\Log\left( \sum_{\beta} \sum_{n \in \BZ} \PTcal_{n,\beta} (-p)^n Q^{\beta} \right). \]
If $2 \nmid \beta$, what we assume from now on, then \eqref{Log toda eqn} becomes:
\begin{equation} \label{12244}
\sum_{n,\beta} \PTcalprim_{n,\beta} (-p)^n
=
\sum_{r,n \geq 0} (-1)^{r-1} [n+r]_{t \tilde{t}} \Omega(r,\beta,n) p^{n} 
+ \sum_{r,n > 0} (-1)^{r-1} [n+r]_{t \tilde{t}} \Omega(r,\beta,n) p^{-n}.
\end{equation}

If $\beta$ is primitive, then the invariants $\Omega(r,\beta,n)$ on the right hand side depend by Theorem~\ref{thm:dependence of DT invariants} only on the square $\beta^2-2rn-r^2$. Hence in this case, also $\PTcalprim_{n,\beta}$ only depends on the square.
This is a remarkable property, since there are usually different monodromy orbits of primitive vectors on an Enriques surface and the linear systems of curves can behave quite differently, see \cite[Sec.4.7]{RefinedEnriques1} for a discussion and examples.
We expect that moreover for any $v=(r,\beta,n)$ with $2 \nmid \beta$ the invariant 
$\Omega(v)$ only depends on $v^2$, and hence that
$\PTcalprim_{n,\beta}$ only depends on $\beta$ via its square $\beta^2$. We will assume this below (If one does not want to assume this, one can just restrict to the case $\beta$ primitive below). We write:
\begin{alignat*}{2}
\PTcalprim_{n,d} & = \PTcalprim_{n,\beta} &  \quad \text{ if } \beta^2 & = 2d,\ 2 \nmid \beta,  \\
\Omega_d & = \Omega(v)&  \quad \text{ if } v^2 & = 2d,\ 2 \nmid \beta.
\end{alignat*}

Stable sheaves in class $v=(r,\beta,n)$ for $2 \nmid \beta$ are supported on one of the four Enriques fibers $Y$ of $Q \to \p^1$,
and hence by \cite[Lemma.3.5]{RefinedEnriques1} also scheme-theoretically.
So the moduli space of stable sheaves on $Q$ with Chern charater $v=(r,\beta,n)$ primitive with $2 \nmid \beta$,
is just 4 copies of a moduli space of stable sheaves on $Y$ (which then split into two components corresponding to the determinant of the sheaves).
In case $v^2=2d$ odd we get
\[ \Omega_{d} = 8 \chi_{t,\tilde{t}}( [\Hilb^{d+1/2}(Y)]^{\vir} ). \]
By G\"ottsche's formula for the Hilbert schemes, we thus have the evaluation:
\[
\sum_{d \text{ half integral}} \Omega_d q^d =
8 \frac{1}{q^{1/2}} \prod_{n \geq 1} \frac{1}{(1- (t \tilde{t})^{-1} q^n) (1-q^n)^{10} (1- t \tilde{t} q^n)}
\]
If $v^2=2d$ is even, we get
\begin{align*}
	\Omega_d & = 8 \chi_{t,\tilde{t}}( [ M((0,\beta_d,0), L) ]^{\mathrm{vir}}) = 8 (-u)^{-(2d+1)} H(M((0,\beta_d,0), L))
\end{align*}
where $\beta_d \in H_2(Y,\BZ)$ is a primitive class of square $2d$
and we use that $M((0,\beta_d,0),L)$ and $M((0,\beta_d,0),L+K_Y)$ are deformation equivalent \cite{Knutsen} so have the same Betti/Hodge numbers, see also \cite[Prop.C.1]{Enriques}.

Rewriting then equation \eqref{12244} in terms of
$\Omega_d$ we find:
\begin{align*}
	& \sum_{d = 0}^{\infty}
	\PTcalprim_{n,d} q^d (-p)^n \\
	= & \left( \sum_{n \text{ half-integral}} \Omega_n q^n \right)
	\sum_{\substack{r \geq 1 \\ r \text{ odd}}} \left( [r]_{t \tilde{t}} q^{r^2/2} + \sum_{n \geq 1} [n+r]_{t \tilde{t}} (p^n + p^{-n}) q^{rn+r^2/2} \right) \\
	& -
	\left( \sum_{n \text{ integral}} \Omega_n q^n \right)
	\left[ \sum_{n \geq 1} [n]_{t \tilde{t}} p^n + \sum_{\substack{r \geq 1 \\ r \text{ even}}} \left( [r]_{t \tilde{t}} q^{r^2/2} + \sum_{n \geq 1} [n+r]_{t \tilde{t}} (p^n + p^{-n}) q^{rn+r^2/2} \right) \right] \\
	= & 
	\left( 8 \frac{1}{q^{1/2}} \prod_{n \geq 1} \frac{1}{(1- (t \tilde{t})^{-1} q^n) (1-q^n)^{10} (1- t \tilde{t} q^n)} \right)
	\frac{\Theta(t \tilde{t}, q^2)}{( (t \tilde{t})^{\frac{1}{2}} - (t \tilde{t})^{-\frac{1}{2}})}
	\frac{\Theta(p (t \tilde{t})^{\frac{1}{2}}, q^2) \Theta(p (t \tilde{t})^{-\frac{1}{2}}, q^2) \eta(q^2)^8}{\Theta(p (t \tilde{t})^{\frac{1}{2}},q) \Theta(p (t \tilde{t})^{-\frac{1}{2}},q) \eta(q)^4} \\
	- & 
	\left( \sum_{n \text{ integral}} \Omega_n q^n \right)
	\frac{\Theta(t \tilde{t},q^2)}{((t \tilde{t})^{\frac{1}{2}}-(t \tilde{t})^{-\frac{1}{2}})} \frac{1}{\Theta((t \tilde{t})^{\frac{1}{2}} p , q^2) \Theta( (t \tilde{t})^{-\frac{1}{2}} p, q^2)} \\
	= & 
	\, 8 \frac{\Theta(t \tilde{t}, q^2)}{ \eta^{12}(q) \Theta(t \tilde{t},q) }
	\frac{\Theta(p (t \tilde{t})^{\frac{1}{2}}, q^2) \Theta(p (t \tilde{t})^{-\frac{1}{2}}, q^2) \eta(q^2)^8}{\Theta(p (t \tilde{t})^{\frac{1}{2}},q) \Theta(p (t \tilde{t})^{-\frac{1}{2}},q) \eta(q)^4} \\
	- & 
	\left( \sum_{n \text{ integral}} \Omega_n q^n \right)
	\frac{\Theta(t \tilde{t},q^2)}{((t \tilde{t})^{\frac{1}{2}}-(t \tilde{t})^{-\frac{1}{2}})} \frac{1}{\Theta((t \tilde{t})^{\frac{1}{2}} p , q^2) \Theta( (t \tilde{t})^{-\frac{1}{2}} p, q^2)} \\
	%\frac{\Theta(p (t \tilde{t})^{\frac{1}{2}}, q^2) \Theta(p (t \tilde{t})^{-\frac{1}{2}}, q^2) \eta(q^2)^8}{\Theta(p (t \tilde{t})^{\frac{1}{2}},q) \Theta(p (t \tilde{t})^{-\frac{1}{2}},q) \eta(q)^4}
\end{align*}
where in the second step we used the Jacobi forms identities
\cite[Prop.2.3,2.4]{RefinedEnriques1}.
% of Propositions~\ref{prop:Jac identity Enriques} and \ref{prop:Jac identity Enriques2},

Specializing further to the Betti realization $t=\tilde{t}=u$ we get
\begin{align*}
\sum_{n,d} \PTcalprim_{n,d}|_{t=\tilde{t}=u} (-p)^n q^d 
= & 8 \frac{\Theta(u^2, q^2)}{ \Theta(u^2,q) }
\frac{ \eta(q^2)^8}{\eta^{16}(q)}
\frac{\Theta(p u, q^2) \Theta(p u^{-1}, q^2)}{\Theta(p u,q) \Theta(p u^{-1},q)} \\
- & 
\frac{ \sum_{n \text{ integral}} \Omega_n|_{t=\tilde{t}=u} q^n }{u-u^{-1}}
\frac{\Theta(u^2,q^2)}{\Theta(u p , q^2) \Theta( u^{-1} p, q^2)}
\end{align*}
We thus obtain the partial evaluation of Gopakumar-Vafa polynomials:
\begin{align*}
\sum_{d \geq 0} \GV_{\beta_d}(u,p) q^d
= & 
8 \frac{(1-u^{-1}p)(1-up)}{(-p)}
\frac{\Theta(u^2, q^2)}{ \Theta(u^2,q) }
\frac{ \eta(q^2)^8}{\eta^{16}(q)}
\frac{\Theta(p u, q^2) \Theta(p u^{-1}, q^2)}{\Theta(p u,q) \Theta(p u^{-1},q)} \\
- & 
\frac{ \sum_{n \text{ integral}} \Omega_n|_{t=\tilde{t}=u} q^n }{u-u^{-1}}
\frac{(1-u^{-1}p)(1-up)}{(-p)}
\frac{\Theta(u^2,q^2)}{\Theta(u p , q^2) \Theta( u^{-1} p, q^2)}
\end{align*}

The Gopakumar-Vafa polynomials satisfy conjecturally
\[ \mathrm{GV}_{\beta}(u,p) =
\sum_{i,j} {^ph}^{i,j}_{\beta} (-1)^{i+j} p^i u^j, \]
where ${^ph}^{i,j}_{\beta}$ are defined in terms of $M_{\beta,1}$.
However, $M_{\beta,1}$ is isomorphic to $4$ copies of the moduli space $M^Y(0,\beta,1) = M^Y((0,\beta,1),L) \sqcup M^Y((0,\beta,1),L+K_Y)$,
and the later two togehter with the Hilbert-Chow morphism are deformation equivalent
% of stable $1$-dimensional sheaves on the Enriques surface $Y$ in class $s+df$.

Altogether we obtain the following concrete conjecture
(equivalent to Conjecture~\ref{conj:perverse Hodge intro}):
\begin{conj} \label{conj:perverse Hodge}
Let $\beta \in H_2(Y,\BZ)$ be an effective curve class on a generic Enriques of square $\beta^2=2d$ with $2 \nmid \beta$, and let $M_d$ be one of the two connected components of the moduli space $M^Y(0,\beta,1)$.
% satisfy:
Then the perverse Hodge numbers ${^ph}^{i,j}(M_d)$ of $M_d$ satisfy:
%be the perverse Hodge numbers of the Hilbert-Chow morphism $M_d \to \p^N$.
%Then:
\begin{multline}
\label{Key equation2}
\sum_{d \geq 0} \sum_{i,j} {^ph}^{i,j}(M_d) (-1)^{i+j} p^i u^j q^d
= 
\frac{(1-u^{-1}p)(1-up)}{(-p)}
\frac{\Theta(u^2, q^2)}{ \Theta(u^2,q) }
\frac{ \eta(q^2)^8}{\eta^{16}(q)}
\frac{\Theta(p u, q^2) \Theta(p u^{-1}, q^2)}{\Theta(p u,q) \Theta(p u^{-1},q)} \\
-
\left( \sum_{d=0}^{\infty} (-u)^{-(2d+1)} H(M_d) \right) \cdot
\frac{(1-u^{-1}p)(1-up)}{(u-u^{-1}) (-p)}
\frac{\Theta(u^2,q^2)}{\Theta(u p , q^2) \Theta( u^{-1} p, q^2)}
\end{multline}
\end{conj}

\begin{rmk}
	We discuss briefly the expected modularity of the generating series
	\begin{align*}
		\sum_{n,d} \PT^{\mathrm{mot},\mathrm{prim}}_{n,d}|_{t=\tilde{t}=u} (-p)^n q^d 
		= & 8 \frac{\Theta(u^2, q^2)}{ \Theta(u^2,q) }
		\frac{ \eta(q^2)^8}{\eta^{16}(q)}
		\frac{\Theta(p u, q^2) \Theta(p u^{-1}, q^2)}{\Theta(p u,q) \Theta(p u^{-1},q)} \\
		+ & 
		\frac{ \sum_{n \text{ integral}} a(n)|_{t=\tilde{t}=u} q^n }{u-u^{-1}}
		\frac{\Theta(u^2,q^2)}{\Theta(u p , q^2) \Theta( u^{-1} p, q^2)}
	\end{align*}
	The first term on the right hand side is a Jacobi form in the elliptic variables $p,u$ of weight $-4$ and index $\begin{pmatrix} -1/2 & 0 \\ 0 & -3/2 \end{pmatrix}$.\footnote{The function $\Theta$ is of index $1/2$ and weight $-1$, and if $f(z,w)$ is of index $\begin{pmatrix} L_{11} & L_{21}^t \\ L_{21} & L_{22} \end{pmatrix}$ for a multivariable $w=(w_1,\ldots,w_l)$, then
		$f(z_1z_2,w)$ and $f(z_1 z_2^{-1}, w)$ are of index
		\[
		\begin{pmatrix}
			L_{11} & L_{11} & L_{21}^{t} \\
			L_{11} & L_{11} & L_{21}^t \\
			L_{21} & L_{21} & L_{22}
		\end{pmatrix},
		\quad 
		\begin{pmatrix}
			L_{11} & -L_{11} & L_{21}^{t} \\
			-L_{11} & L_{11} & -L_{21}^t \\
			L_{21} & -L_{21} & L_{22}
		\end{pmatrix},
		\]
		respectively.
		Moreover, sending $f(p,q)$ to $f(p,q^2)$ (or $f(p^2,q)$) does not change the weight but divides the index by $2$ (or multiplies the index by $4$ respectively), see \cite{EZ, HilbHAE}.
		%, and sending $f(p,q)$ to $f(p^2,q)$ does not change the weight and multiplies the index by $2^2=4$.
	}
	Since we expect the above to be an equality of Jacobi forms of the same weight and index, we thus expect that
	\[
	\frac{\sum_{n \text{ integral}} a(n)|_{t=\tilde{t}=u} q^n}{(u-u^{-1})^{-1}}
	=
	8 \sum_{d=0}^{\infty} (-u)^{-(2d+1)} H(M((0,s+nf,0),L) q^d \]
	is a Jacobi form of weight $-5$ and index $-2$.
	Unfortunately, its not clear what is the group under which it satisfies the Jacobi form identities and also not the pole structure, so the very minimal data we have is not enough to fix this.  \qed
\end{rmk}

\subsection{Proof of Proposition~\ref{prop:asymptotics}}
We state a couple of basic lemmas, that will lead to the proof of
Proposition~\ref{prop:asymptotics}.

\begin{lemma} \label{lemma:asymptotics1}
	If \Cref{conj:perverse Hodge} holds and $i,j < -d/2 - 1$, then
	\[ {^ph}^{i,j}(M_d) (-1)^{i+j}
	=
	\mathrm{Coeff}_{p^i u^j q^d}
	\left[
	\frac{(1-u^{-1}p)(1-up)}{(-p)}
	\frac{\Theta(u^2, q^2)}{ \Theta(u^2,q) }
	\frac{ \eta(q^2)^8}{\eta^{16}(q)}
	\frac{\Theta(p u, q^2) \Theta(p u^{-1}, q^2)}{\Theta(p u,q) \Theta(p u^{-1},q)} \right]
	\]
\end{lemma}
\begin{proof}
By direct check, the second term does not contribute in this range.
\end{proof}

\begin{lemma}
	Consider the function
\[ F(u,p,q)
=
\frac{(1-u^{-1}p)(1-up)}{(-p)}
\frac{\Theta(u^2, q^2)}{ \Theta(u^2,q) }
\frac{ \eta(q^2)^8}{\eta^{16}(q)}
\frac{\Theta(p u, q^2) \Theta(p u^{-1}, q^2)}{\Theta(p u,q) \Theta(p u^{-1},q)}. \]
For fixed $i,j$, the $p^{i-d-1} u^{j-d} q^d$ coefficient of $F$ is independent of $d$ for $d > \frac{3}{4} (i+j)$. Denote it (up to sign) by
$a_{ij} := (-1)^{i+j+1} \mathrm{Coeff}_{p^{i-d-1} u^{j-d} q^d}(F)$.
It is given by
\[
	\sum_{i,j \geq 0} a_{ij} x^i y^j
=
(1-xy) \prod_{n \geq 1} \frac{1}{(1-x^{n+1} y^{n-1}) (1-x^{n-1} y^{n+1}) (1-x^n y^n)^{10)}}
\]
\end{lemma}
\begin{proof}
We follow \cite[Cor.2.11]{Goettsche}.
Let $\tilde{F}(u,p,q) = (-p) F(u,p, upq)$. Then
\[
\mathrm{Coeff}_{p^{i-d-1} u^{j-d} q^d}(F) = -\mathrm{Coeff}_{p^i u^j q^d}(\tilde{F}(u,p,q)).
\]
We define $G = (1-q) \tilde{F}$.
Then the coefficient $u^i p^j q^d$ of $G$ vanishes whenever $d > \frac{3}{4} (i+j)$, so arguing as in \cite{Goettsche} we see that for fixed $i,j$ the $p^i u^j q^d$ coefficient of $\tilde{F}$ becomes independent of $d$ for all $d > \frac{3}{4}(i+j)$, lets say it is $b_{ij}$ (which is $a_{ij}$ up to the sign $(-1)^{i+j}$). But then as in \cite{Goettsche} $\sum_{i,j}b_{ij} u^i p^j b_{ij} = G(u,p,q=1)$, which gives the result.
% Moreover, the asymptotic is given precisely by $G(u,p,q=1)$ which shows the result.
%$\sum_{i,j} (-1)^{i+j} a_{ij} = G(u,p,q=1)$ which gives the result.
%We first rewrite
%\begin{multline*}
%	F(u,p,q) = \Bigg[
%	\frac{(1-u^{-1}p)(1-up)}{(-p)} \prod_{m \geq 1} \frac{1}{(1-q^m)^{8}} \\
%	\times 
%	\prod_{\substack{m \geq 1 \\ m \text{ odd}}} 
%	\frac{1}{(1-u^{-2} q^m) (1-u^2 q^m) (1-up q^m) (1-u p^{-1} q^m)(1-u^{-1} p q^m) (1-u^{-1} p^{-1} q^m) (1-q^m)^2} 
%	\Bigg]
%\end{multline*}
%
%Then we argue as in Goettsche.
\end{proof}

\section{Open questions}
\label{sec:open questions}
\subsection{Relationship $K_Y$ and $Q$}
As proven in \cite{Enriques}, the (unrefined) generalized DT invariants of $Q$ in fiber classes are related to the (unrefined) Vafa-Witten invariants of $Y$ by
	\[ \VW^{\mathrm{unref}}(v) = \frac{1}{4} \chi(\DTcal(v)). \]
	% see \cite{Enriques}.
	We may ask whether this equality can be refined:

\begin{question} How is the $\chi_{t}$ specialization of the motivic invariants of $Q$, $\DTcal(v)|_{\tilde{t}=1}$ related to the Nekrasov-Okounkov refined Vafa-Witten invariant $\VW(v)$ of $Y$ (studied in \cite{RefinedEnriques1}?
	\end{question}
	
	For $v=(r,\beta,n)$ with $r$ is odd or $2 \nmid \beta$ we have $v \notin \pi_{\ast} H^{\ast}(X,\BZ)$, where $\pi : X \to Y$ is the covering K3 surface. Thus stable sheaves on $Q$ with this Chern character are supported on the Enriques fibers of $Q \to \p^1$ and one should  expect:
	\begin{equation} \label{VW=DT} \VW(v) = \frac{1}{4} \DTcal(v)|_{\tilde{t}=1} \end{equation}
		
	However, Corollary~\ref{cor:DTcal in fiber classes} and \cite[Prop.4.7]{RefinedEnriques1} shows that \eqref{VW=DT} fails in general, e.g. one has:
	\[
	\DTcal(2,2f,0)|_{\tilde{t}=1} = \frac{t^{-1} - 2 + t}{[2]_t}, \quad \VW(2,2f,0)=0.
	\]
	
\subsection{Dependence on $\beta$}
For $\beta \in H_2(Y,\BZ)$ let $\GV_{\beta}(p,q)$ be the refined Gopakumar-Vafa polynomial of $Q$.
If $2 \nmid \beta$, we conjectured that $\GV_{\beta}$ should only depend on the square.
For $2|\beta$ we only have minimal amount of computations in 
Section~\ref{subsec:fiber curve classes}.
Still it is natural to ask:
\begin{question}
	If $2|\beta$, does $\GV_{\beta}$ also depend only on the square of the class $\beta$?
\end{question}


\begin{thebibliography}{10}

\bibitem{Beckmann}
T. Beckmann,
{\em Birational geometry of moduli spaces of stable objects on Enriques surfaces}, Selecta Math. (N.S.) {\bf 26} (2020), no. 1, Paper No. 14, 18 pp.


\bibitem{Behrend}
K. Behrend,
{\em 
	Donaldson-Thomas type invariants via microlocal geometry},
Ann. of Math. (2) {\bf 170} (2009), no. 3, 1307--1338.




\bibitem{BBS}
K. Behrend, J. Bryan, B. Szendroi,
{\em Motivic degree zero Donaldson-Thomas invariants},
Invent. Math. {\bf 192} (2013), no. 1, 111--160.


\bibitem{BBBJ}
O. Ben-Bassat, C. Brav, V. Bussi, D. Joyce,
{\em A `Darboux theorem' for shifted symplectic structures on derived Artin stacks, with applications},
Geom. Topol.{\bf 19} (2015), no.3, 1287--1359.

\bibitem{BJM} V. Bussi, D. Joyce, and S. Meinhardt, {\em On motivic
	vanishing cycles of critical loci},
J. Algebraic Geom. {\bf 28} (2019), no. 3, 405--438.



\bibitem{CKK}
J. Choi, S. Katz, A. Klemm,
{\em The refined BPS index from stable pair invariants},
Comm. Math. Phys. {\bf 328} (2014), no. 3, 903--954.

\bibitem{DM15}
B. Davison, S. Meinhardt, 
{\em Motivic Donaldson-Thomas invariants for the one-loop
quiver with potential}, Geom. Topol. {\bf 19} (2015), no. 5, 2535--2555.

\bibitem{dCM}
M. A. de Cataldo, L. Migliorini,
{\em The decomposition theorem, perverse sheaves and the topology of algebraic maps}, Bull. Amer. Math. Soc. (N.S.){\bf 46}(2009), no.4, 535--633.

\bibitem{Dec22a}
P. Descombes,
{\em Cohomological DT invariants from localization},
J. Lond. Math. Soc. (2) {\bf 106} (2022), no. 4, 2959--3007.

\bibitem{DIW}
A. Doan, E.-N. Ionel, T. Walpuski,
{\em The Gopakumar-Vafa finiteness conjecture}, arXiv:2103.08221


\bibitem{EZ} 
M. Eichler, D. Zagier, {\em The theory of Jacobi forms}, Progress in Mathematics, 55. Birkhäuser Boston, Inc., Boston, MA, 1985. v+148 pp.


\bibitem{Goettsche}
L. G\"ottsche,
{\em The Betti numbers of the Hilbert scheme of points on a smooth projective surface},
Math. Ann. {\bf 286} (1990), no. 1-3, 193--207. 


\bibitem{GLM}
S. M. Gusein-Zade, I. Luengo, A. Melle-Hernandez,
{\em On generating series of classes of equivariant Hilbert schemes of fat points},
Mosc. Math. J.{\bf 10} (2010), no.3, 593--602, 662.


\bibitem{HST}
S. Hosono, M.-H. Saito, A. Takahashi, 
{\em Relative Lefschetz action and BPS state counting}, Internat.
Math. Res. Notices 2001, no. {\bf 15}, 783--816.

\bibitem{HKKrefined}
M.-x. Huang, S. Katz, A. Klemm,
{\em  Towards refining the topological strings on compact Calabi-Yau 3-folds},
J. High Energy Phys.(2021), no. 3, Paper No. 266, 87 pp.




\bibitem{JS} D. Joyce, Y. Song,
{\em A theory of generalized Donaldson-Thomas invariants},
Mem. Amer. Math. Soc. {\bf 217} (2012), no. 1020, iv+199 pp.

\bibitem{JU}
D. Joyce, M. Upmeier,
{\em Orientation data for moduli spaces of coherent sheaves over Calabi-Yau 3-folds}, Adv. Math. {\bf 381} (2021), Paper No. 107627, 47 pp.

\bibitem{KKP}
S. Katz, A. Klemm, R. Pandharipande,
{\em On the motivic stable pairs invariants of K3 surfaces},
Progr. Math., {\bf 315}
Birkhäuser/Springer, [Cham], 2016, 111--146.

\bibitem{KiemLi}
Y. H. Kiem, J. Li, {\em Categorification of Donaldson–Thomas invariants via perverse sheaves},
arXiv:1212.6444v5.


\bibitem{Knutsen}
A. L. Knutsen,
{\em On moduli spaces of polarized Enriques surfaces},
J. Math. Pures Appl. (9) {\bf 144} (2020), 106--136.


\bibitem{KPS}
Y. Kononov, W. Pi, J. Shen,
{\em Perverse filtrations, Chern filtrations, and refined BPS invariants for local P2}, Adv. Math. 433 (2023), Paper No. 109294, 29 pp.


\bibitem{KS}
M. Kontsevich, Y. Soibelman,
{\em Cohomological Hall algebra, exponential Hodge structures and motivic Donaldson-Thomas invariants},
Commun. Number Theory Phys. {\bf 5} (2011), no. 2, 231--352.


\bibitem{Leigh}
O. Leigh,
{\em Instantons and multibananas: Relating elliptic genus and cohomological Donaldson-Thomas theory},
arXiv:2404.16018


\bibitem{MP}
D. Maulik, R. Pandharipande,
{\em A topological view of Gromov-Witten theory},
Topology {\bf 45} (2006), no. 5, 887--918. 


\bibitem{MS}
D. Maulik, J. Shen,
{\em Cohomological $\chi$-independence for moduli of one-dimensional sheaves and moduli of Higgs bundles}, Geom. Topol.{\bf 27}(2023), no.4, 1539--1586.


\bibitem{MT}
D. Maulik, Y. Toda,
{\em Gopakumar-Vafa invariants via vanishing cycles},
Invent. Math. {\bf 213} (2018), no. 3, 1017--1097.



\bibitem{MaulikYun}
D. Maulik, Z. Yun,
{\em Macdonald formula for curves with planar singularities}, J. Reine Angew. Math. {\bf 694} (2014), 27--48.



\bibitem{MigShende1}
L. Migliorini, V. Shende,
{\em A support theorem for Hilbert schemes of planar curves},
J. Eur. Math. Soc. (JEMS) {\bf 15} (2013), no. 6, 2353--2367.

\bibitem{MigShende2}
L. Migliorini, V. Shende, F. Viviani,
{\em A support theorem for Hilbert schemes of planar curves, II},
Compos. Math. {\bf 157} (2021), no. 4, 835--882.

\bibitem{MMNS}
A. Morrison, S. Mozgovoy, K. Nagao, B. Szendroi, 
{\em Motivic Donaldson-Thomas invariants of the conifold and the refined topological vertex}, Adv. Math. {\bf 230} (2012),
no. 4-6, 2065--2093.

\bibitem{NO}
N. Nekrasov, A. Okounkov,
{\em Membranes and sheaves},
Algebr. Geom. {\bf 3} (2016), no. 3, 320--369.


\bibitem{NY}
H. Nuer, K. Yoshioka, {\em MMP via wall-crossing for moduli spaces of stable sheaves on an Enriques
	surface},
Adv. Math. {\bf 372} (2020), 107283, 119 pp


\bibitem{Enriques} G. Oberdieck,
{\em Curve counting on the Enriques surface and the Klemm-Marino formula},
to appear in JEMS, arXiv:2305.11115


\bibitem{RefinedEnriques1}
G. Oberdieck,
{\em Towards refined curve counting on the Enriques surface I: K-theoretic refinements}, Preprint.


\bibitem{HilbHAE}
G. Oberdieck,
{\em  Holomorphic anomaly equations for the Hilbert scheme of points of a K3 surface},
to appear in Geom. Topol.,
arXiv:2202.03361



\bibitem{PT}
R. Pandharipande, R. P. Thomas, {\em Curve counting via stable pairs in the derived category}, Invent. Math. {\bf 178} (2009), no. 2, 407--447.

\bibitem{PSSZ}
W. Pi, J. Shen, F. Si, F. Zhang,
{\em Cohomological stabilization, perverse filtrations, and refined BPS invariants for del Pezzo surfaces},
Preprint.

\bibitem{Sacca}
G. Sacc\`a,
{\em Relative compactified Jacobians of linear systems on Enriques surfaces}, Trans. Amer. Math. Soc. {\bf 371} (2019), no. 11, 7791--7843.


\bibitem{SY}
J. Shen, Q. Yin,
{\em Topology of Lagrangian fibrations and Hodge theory of hyper-Kähler manifolds},
Duke Math. J. {\bf 171} (2022), no. 1, 209--241.

\bibitem{Toda2}
Y. Toda,
{\em Stability conditions and curve counting invariants on Calabi-Yau 3-folds}, Kyoto J. Math. {\bf 52} (2012), no.1, 1--50.

\end{thebibliography}
\end{document}